\documentclass[10pt,twoside]{siamltex}
\usepackage{amsfonts,epsfig}
\usepackage{array}
\usepackage[leqno]{amsmath}
\usepackage{multirow}
\usepackage{mathrsfs}
\usepackage[colorlinks]{hyperref}
\newtheorem{remark}[theorem]{Remark}

\newtheorem{example}[theorem]{Example}

\def \F  {\mathcal{F}}
\def \vec {\mathrm{vec}}
\def \M  {\mathcal{M}}

\begin{document}

%  Leave these commented lines here
% \input{elaheader-volx-xx.tex}
% \setcounter{page}{1}

% \renewcommand{\thefootnote}{\fnsymbol{footnote}}
% \renewcommand{\thefootnote}{\arabic{footnote}}
% \renewcommand{\theequation}{\thesection.\arabic{equation}}

\bibliographystyle{plain}
\title{
On the condition number theory of the equality constrained indefinite least squares problem}
% Leave blank; editors will write the exact dates above

\author{
Shaoxin Wang\thanks{Corresponding author. School of Statistics, Qufu Normal University, Qufu, 273165, P. R. China
(shxwang@qfnu.edu.cn or shwangmy@163.com). Supported by a project of Shandong Province Higher Educational Science and Technology Program (Grant No. J17KA160).}
% Remember to put \and between any two authors
\and
Hanyu Li\thanks{College of Mathematics and Statistics, Chongqing University, Chongqing, 401331, P. R. China (hyli@cqu.edu.cn, yh@cqu.edu.cn). Supported by the National Natural Science Foundation of China (Grant Nos. 11671060, 11671059).}
\and
Hu Yang\footnotemark[2]}
% Note that \footnotemark[3]} is used for the third author
% because of the same affiliation for the second and third authors.
% If the same affiliation is to be used for the first and second authors,
% \footnotemark[2] should be used instead of \thanks{} for the second author.

% Authors and running title to go on top of each page
\pagestyle{myheadings}
\markboth{S.X.\ Wang, H.Y.\ Li, and H.\ Yang}{Condition Number Theory of the EILS Problem}
\maketitle

\begin{abstract}
  In this paper, within a unified framework of the condition number theory we present the explicit expression of the \emph{projected} condition number of the equality constrained indefinite least squares problem. By setting specific norms and parameters, some widely used condition numbers, like the normwise, mixed and componentwise condition numbers follow as its special cases. Considering practical applications and computation, some new compact forms or upper bounds of the projected condition numbers are given to improve the computational efficiency. The new compact forms are of particular interest in calculating the exact value of the 2-norm projected condition numbers. When the equality constrained indefinite least squares problem degenerates into some specific least squares problems, our results give some new findings on the condition number theory of these specific least squares problems. Numerical experiments are given to illustrate our theoretical results.
\end{abstract}


%%%%%%%%%%%%%%%%%%%%%%%%%%%%%%%%%%%%%%%%%%%%%%%%%%%%%%%%%%%%%
\section{Introduction}
\label{intro}
The equality constrained indefinite least squares (EILS) problem  can be stated as follows
\begin{eqnarray}
\label{ILSE}
{\rm EILS:}\quad \mathop {\min }\limits_{x \in\mathbb{R} {^n}} {(b - Ax)^T}J(b - Ax)\quad \textrm{ subject to } Bx = d,
\end{eqnarray}
where $A \in {\mathbb{R}^{m \times n}}$ is full column rank, $B \in {\mathbb{R}^{s \times n}}$, $b \in \mathbb{R}^{m}$, $d \in {\mathbb{R}^s}$, and $J$ is a signature matrix defined as
$
J = \left[ {\begin{array}{*{20}{c}}
{{I_p} }&0\\
0&{- {I_q}}
\end{array}} \right]$ with $ p+q=m$. Herein, we let $\mathbb{R}^{m \times n}$ and $\mathbb{R}^{p}$ stand for the sets of $m\times n$ real matrices and $p$ dimensional column real vectors, respectively. $A^T$ denotes the transpose of $A$, and $I_s$ denotes the identity matrix of order $s$. By varying the constraints $Bx=d$ and the matrix $J$, we can get some specific least squares (LS) problems from \eqref{ILSE}. For example, if we remove $Bx=d$, then the indefinite least squares (ILS) problem follows,  which can be used to solve the total least squares (TLS) problem \cite{Huf} and $H^{\infty}$-smoothing \cite{Has}. If we set $J=I_m$, then the equality constrained least squares (ELS) problem follows, which can be applied to analyze the large-scale structures in engineering \cite{Bar88b}. Thus, the EILS problem and its special cases have attracted many researchers to study its algorithms, error analysis, and perturbation theory (cf. \cite{Boa,Boja03, Cha, Liu11, Liu10b, Masta14, Mast14b, Wang}). Our discussion is performed under the following assumptions given in \cite{Boja03, Pate02}
\begin{eqnarray}
\label{asump1}
\mathrm{rank}(B)=s \textrm{, and }\; x^T(A^TJA)x>0\; \textrm{for all nonzero}\; x\in \mathcal{N(B)},
\end{eqnarray}
where $\mathcal{N(B)}$ denotes the null space of $B$. The first condition implies that the constraint equations admit a solution. The second one, imposing the positive definiteness of $A^TJA$ on $\mathcal{N(B)}$, ensures the EILS problem \eqref{ILSE} has a unique solution. Assumption \eqref{asump1} also implies
$p\geq n-s$. Under the assumption \eqref{asump1}, the solution to the EILS problem satisfies the following augmented system
\begin{eqnarray}
\label{agmt.eq}
  \begin{bmatrix}
     0 & 0 & B \\
     0 & J & A \\
     B^T & A^T & 0 \\
   \end{bmatrix}\begin{bmatrix}
                  \lambda \\
                  Jr \\
                  x \\
                \end{bmatrix}
   &=& \begin{bmatrix}
          d \\
          b \\
          0 \\
        \end{bmatrix},
\end{eqnarray}
where $r=b-Ax$ and $\lambda=-(BB^T)^{-1}BA^TJr$ is the vector of Lagrange multipliers. Detailed discussion of the existence and uniqueness of the solution to the EILS problem is referred to \cite{Pate02}.

To measure the sensitivity of solution to a small perturbation in the input data, Rice \cite{Rice} developed the general theory of condition number. Let $\phi: \mathbb{R}^p\rightarrow \mathbb{R}^q$ be a continuous and Fr\'{e}chet differentiable map defined on an open set $Dom(\phi)$. For $x_0\in Dom(\phi)$, $x_0\neq 0$, such that $\phi(x_0)\neq 0$ and a small neighbourhood $U(x_0,\epsilon)=\left\{x\in \mathbb{R}^p: \|x-x_0\|\leq \epsilon\right\}\in Dom(\phi)$, according to \cite{Rice} the relative condition number of $\phi$ at $x_0$ is given by
\begin{equation}\label{cdRice}
  \kappa_{\phi}(x_0)=\lim_{\epsilon\rightarrow 0}\mathop {\sup }\limits_{\mathop {x \in U(x_0,\epsilon)}\limits_{x \ne x_0}}\frac{\|\phi(x)-\phi(x_0)\|\|x_0\|}{\|\phi(x_0)\|\|x-x_0\|}=\frac{\|D\phi(x_0)\|\|x_0\|}{\|\phi(x_0)\|},
\end{equation}
where $\|\cdot\|$ denotes a generic vector norm on $\mathbb{R}^p$ and $\mathbb{R}^q$, and $D\phi(x_0)$ is the Fr\'{e}chet derivative of $\phi$ at $x_0$. Since \eqref{cdRice} may ignore the data structure or scaling in the data, Gohberg and Koltracht \cite{Gohb93} proposed the following mixed and componentwise condition numbers. Let $U^0(x_0,\epsilon)=\left\{x\in\mathbb{R}^{p}: |x_i-x_{0i}|<\epsilon |x_{0i}|\right\}$ and $\epsilon$ be small enough such that $U^0(x_0,\epsilon)\in Dom(\phi)$. Then the mixed condition number is defined as
\begin{equation}\label{cdmix}
  \kappa_{m\phi}(x_0)=\lim_{\epsilon\rightarrow 0}\mathop {\sup}\limits_{\mathop {x \in U^0(x_0,\epsilon)}\limits_{x \ne x_0}}\frac{{{{\left\| {\phi(x) - \phi(x_0)} \right\|}_\infty }}}{{{{\left\| {\phi(x_0)} \right\|}_\infty d(x,x_0) }}}=\frac{\left\||D\phi(x_0)||x_0|\right\|_{\infty}}{\|\phi(x_0)\|_{\infty}},
\end{equation}
and the componentwise condition number is given by
\begin{equation}\label{cdcopt}
  \kappa_{c\phi}(x_0)=\lim_{\epsilon\rightarrow 0}\mathop {\sup}\limits_{\mathop {x \in U^0(x_0,\epsilon)}\limits_{x \ne x_0}}\frac{{{d(\phi(x),\phi(x_0))}}}{{{ d(x,x_0) }}}=\left\|\frac{|D\phi(x_0)||x_0|}{|\phi(x_0)|}\right\|_{\infty},
\end{equation}
where $\|x\|_{\infty}=\max_{i=1,\cdots,p}|x_i|$, $d(x,x_0)$ is componentwise relative distance between vectors and given by $d(x,x_0)=\max_{i=1,\cdots,p}|x_i-x_{0i}|/|x_{0i}|$ with $x_{0i}\neq 0$, and $|x|$ is to take the absolute value of elements in $x$. It should be noted that, from \eqref{cdmix} and \eqref{cdcopt}, the elements of $\phi(x_0)$ and $x_0$ are required to be nonzero. This may limit its applications. By redefining the componentwise relative distance as
\begin{equation*}\label{divid}
  d(x,x_0)=\max_{i=1,\cdots,p}\left|x_{0i}^{\ddag}\right||x_i-x_{0i}|
\end{equation*}
with $x_{0i}^{\ddag}=\left\{
                       \begin{array}{ll}
                         1/x_{0i}, & \hbox{$x_{0i}\neq 0$;} \\
                         1, & \hbox{$x_{0i}= 0$.}
                       \end{array}
                     \right.$,
Xie et al. \cite{Xieli} proposed a modified version of the mixed and componentwise condition numbers. In their setting,  the values of mixed and componentwise condition numbers are always finite. From \eqref{cdcopt}, we may bound the forward error by
 \begin{equation}\label{rlthumb}
 d\left(\phi(x),\phi(x_0)\right)\leq \kappa_{c\phi}(x_0) d(x,x_0).
 \end{equation}
For illustration, we set the elements in $x$ to be nonzero, then we may say that the forward error is bounded by $\kappa_{c\phi}(x_0)$ multiplied by the relative backward error. But, just from \eqref{rlthumb} we can not tell $d(\phi(x),\phi(x_0))$ gives relative or absolute forward error since the elements in $\phi(x_0)$ may equal zero. To remedy this drawback, we adopt a modified definition of condition number, the \emph{projected} condition number, which can be used to provide a unified framework of the condition number theory and give more elaborate error analysis of the EILS problem.  More researches on the normwise, mixed, and componentwise condition numbers of LS problems can be found in \cite{Bab,BabG09,Cucker07,Diao17,Geu,JiaL13,LiJ11,Li14,ZhMW17}.

The condition number theory of the EILS problem has been studied in the literature. Bojanczyk et al.~\cite{Boja03} gave an upper bound of the normwise condition number. Substituting the equality constraints with LS constraints, Liu and Wang \cite{Liu10b} and Wang \cite{Wang} reconsidered its perturbation theory and also gave some upper bounds. But the explicit expression of the normwise condition number has not been given. Moreover, these upper bounds contain Kronecker product which may make the computation expensive. In \cite{Li14}, the authors gave the explicit expressions of the mixed and componentwise condition numbers, but these condition numbers can be infinite due to their definitions. In this paper, with the projected condition number, we present a generic form of the projected condition number for the EILS problem. The generic form has its generality in covering  the popular normwise, mixed and componentwise condition numbers as its special cases. Since the generic form contains Kronecker product and is not applicable for practical use, we propose some strategies to facilitate the computation of the projected condition number with respect to different settings. Some numerical experiments are also given to illustrate our theoretical results.

The rest of the paper is organized as follows. Section \ref{sec.2} presents some notation and preliminaries. In Section \ref{sec.3}, we present the main results on the condition number theory of the EILS problem. Section \ref{sec.4} contains some new results on the condition numbers of several specific LS problems. In Section \ref{sec.5} we provide some numerical examples to illustrate the results given in Section \ref{sec.3}. Finally, we present the concluding remark of the whole paper.

\section{Preliminaries}
\label{sec.2}
We first introduce some notation. For a vector $b=[b_1,\cdots,b_p]^T\in \mathbb{R}^{p}$, $\|b\|_2=\sqrt{\sum_{i=1}^pb_i^2}$, $\|b\|_{\infty}=\max_{i=1:m}\{|b_i|\}$, and $\|b\|_{1}=\sum_{i=1}^{m}|b_i|$ with $|b_i|$ being the absolute value of $b_i$. For vectors $a\in \mathbb{R}^{p}$ and $b$, we define the following entry-wise division between two vectors:
\begin{eqnarray}
\label{eq.cdivis}
\frac{a}{b}= {\rm diag}(b^{\ddag})a,
\end{eqnarray}
where ${\rm diag}(b^{\ddag})$ is a diagonal matrix with elements $b_1^\ddag,\cdots,b_p^\ddag$ on its diagonal. For a real number $c\in \mathbb{R}$, $c^{\ddag}$ is defined as
$c^{\ddag}=\left\{\begin{array}{ll}
1/c, & \hbox{$c\neq 0$,} \\
1, & \hbox{$c= 0$ .}
\end{array}
\right.$
It should be noted that $\left(b^{\ddag}\right)^{\ddag}=b$ holds only in the case that $b_i\neq 0$ for $i=1,\cdots,p$.

For a matrix $A\in\mathbb{R}^{m\times n}$, $\|A\|_2$ is the spectral norm, $\|A\|_F$ denotes its Frobenius norm, and $\|A\|_{\max}=\max_{i,j}|a_{ij}|$. Let $A=[a_1,\cdots, a_n]$ with $a_i\in \mathbb{R}^{m}$, we use $\mathrm{vec}(\cdot)$ operator to vectorize matrix and ${\rm vec}(A)=[a_1^T,\cdots,a_n^T]^T\in \mathbb{R}^{mn}$. The Hadamard product of $A =[a_{ij}]$ and $C=[c_{ij}] \in{\mathbb{R}^{m \times n}}$ is defined as $A\circ C=[a_{ij}c_{ij}]\in{\mathbb{R}^{m \times n}}$ \cite[p.~298]{Horn91}. The Kronecker product between $A$ and $B\in{\mathbb{R}^{p \times q}}$ is defined as $A\otimes B = [a_{ij}B]\in \mathbb{R}^{mp\times nq}$ \cite[p.~22]{Grah82}, and we get the following results on the relationships between vec$(\cdot)$ and Kronecker product from \cite[Ch.~4]{Horn91}.
\begin{eqnarray}
%&&(A \otimes B)(C \otimes D)=(AC) \otimes (BD),\label{eq.kron1}\\
&&{\rm vec}(AXB) = \left({B^T} \otimes A\right){\rm vec}(X),\label{eq.kron2}\\
&&\Pi_{mn} {\rm vec}(A) = {\rm vec}({A^T}),\;\Pi_{pm} (A \otimes B) \Pi_{nq}= (B \otimes A),\label{eq.PI1}\\
&& (A \otimes B)\Pi_{nq}= (B \otimes A) \textrm{, when }m=1,\label{eq.PI2}
\end{eqnarray}
where $\Pi_{st} \in {\mathbb{R}^{st \times st}}$ is a vec-permutation matrix and only depends on $s$ and $t$ \cite[p.~32]{Grah82}.

In order to define the projected condition number, we consider the following projection map
\begin{eqnarray}\label{prjmap}
  \F_{L}:\mathbb{R}^{p}&\rightarrow& \mathbb{R}^{k}\nonumber\\
          x&\rightarrow& L^T\F(x),
\end{eqnarray}
where $L\in\mathbb{R}^{q\times k}$ with $rank(L)=k$, and $\F(x)\in\mathbb{R}^{q}$. $L$ can be treated as a projection operator to project $\F(x)$ onto a lower dimension space. The idea originates from \cite{Cao03} in estimating the error of some elements in the solution of linear system, and has been applied to investigate the condition numbers of the TLS problem \cite{Bab, DiaoS16} and the ELS problem \cite{LiW17} with notation \emph{partial} condition number. Since the power of $L$ is not just to give the condition number of \emph{partial} elements in  the solution but also can be used to \emph{project} the solution onto another lower dimension space \cite{Cao03},  the definition of projected condition number is given as follows, which is also used in \cite{WangY17} by the authors.

\begin{definition}[Projected condition number]
\label{def.Projcd}
Let $\F: \mathbb{R}^{p}\rightarrow \mathbb{R}^{q}$ be a continuous map defined on an open set $Dom(\F)\in\mathbb{R}^{p}$, the domain of definition of $\F$, $L\in\mathbb{R}^{q\times k}$ with $rank(L)=k$. Then the projected condition number of $\F$ at $x\in Dom(\F)$ with respect to $L$ is defined by
\begin{eqnarray}
 \kappa_{L\F}(x)=\lim_{\delta\rightarrow 0}\sup_{\left\|\beta\circ{\Delta x}\right\|_\mu\leq \delta}\frac{\left\|\xi_L\circ\left(\F_L(x+\Delta x)-\F_L(x)\right)\right\|_{\nu}}{\left\|\chi\circ{\Delta x}\right\|_{\mu}},
\end{eqnarray}
where $\F_L(\cdot)$ is defined by \eqref{prjmap}, $\xi_L\in \mathbb{R}^{k}$, $\chi\in \mathbb{R}^{p}$ with $\chi_i\neq 0$ , and $\|\cdot\|_\mu$ and $\|\cdot\|_\nu$ are two vector norms defined on $\mathbb{R}^{p}$ and $\mathbb{R}^{k}$, respectively.
\end{definition}

When the map $\F$ is Fr\'{e}chet differentiable at $x$, we get the following theorem and its proof can be found in \cite{WangY17}.
\begin{theorem}[\cite{WangY17}]\label{Thm.EXFPCD}
With Definition \ref{def.Projcd}, when the map $\F$ is Fr\'{e}chet differentiable at $x$, the projected condition number of $\F$ with respect to $L$ is given by
\begin{equation*}
\kappa_{L\F}(x)=\left\|\xi_L\circ\left({L^TD\F(x)\mathrm{diag}(\chi^{\ddagger})}\right)\right\|_{\mu\nu},
\end{equation*}
where $D\F(x)$ is the Fr\'{e}chet derivative of $\F$ at $x$.
\end{theorem}
\begin{remark}
\label{rmk1}
\rm
Actually, the parameters $\xi_L$ and $\chi$ can be chosen as positive real numbers instead of vectors. In this case the Hadamard product reduces to regular product between scalar and vector. It has been shown in \cite{WangY17} that Definition \ref{def.Projcd} is the generalization of several popular condition numbers. Let $L=I_q$. When $\mu=\nu=2$, $\chi={1}/{\|x\|_2}$ with $x\neq0$, and $\xi_L={1}/{\|\F(x)\|_2}$ with $\F(x)\neq0$, the relative normwise condition number given in \cite{Rice,Geu} follows; when $\mu=\nu=\infty$, $\chi=[{1}/{x_1},\cdots,{1}/{x_p}]^T$ with $x_{i}\neq0$ and $\xi_L=1/\|\F(x)\|_{\infty}$ with $\F(x)\neq0$ ($\xi_L=[1/\F(x)_1,\cdots,1/\F(x)_q]$ with $\F(x)_i\neq 0$), Definition \ref{def.Projcd} reduces to the mixed (componentwise) condition number given in \cite{Gohb93}. In addition, when the parameters are nonzero, Definition \ref{def.Projcd} is equivalent to the partial condition number \cite[Definition~2.1]{LiW16aX} with $\xi_L=\xi_L^{\ddag}$ and $\chi=\chi^{\ddag}$. Compared with \cite[Definition~2.1]{LiW16aX}, the only difference is that Definition \ref{def.Projcd} allows the zero elements in $x$ to be perturbed in any sense, which does not hold in \cite{LiW16aX,Xieli}. At last, Theorem \ref{Thm.EXFPCD} gives the explicit expression of the projected condition number and largely reduces the difficulty in calculating the value of the projected condition number.
\end{remark}
\section{The condition number of the EILS problem}
\label{sec.3}
For the simplicity of discussion we assume that $A^TJA$ is positive definite\footnote{without this assumption, the following discussion can also be carried out with generalized Moore-Penrose inverse in indefinite inner product space, just like \cite{Li14}.} for the EILS problem, and we define the following map $\F$ from the data space $(A,B,b,d)$ to the solution space $x$:
\begin{align}
  \F:\; \mathbb{R}^{m\times n}\times \mathbb{R}^{s\times n}\times \mathbb{R}^{m}\times \mathbb{R}^{d} &\rightarrow  \mathbb{R}^{n} \nonumber\\
  (A,B,b,d) &\rightarrow  x=M^{-1}B^TN^{-1}d-P^TM^{-1}A^TJb,\label{def_mapf}
\end{align}
where $M=A^TJA$, $N=BM^{-1}B^T$, $P=B^TN^{-1}BM^{-1}-I_n$, and $x$ is the unique solution to the EILS problem. To measure the magnitude of perturbations in the data space, we define the following product norm
\begin{eqnarray}
 \label{def.pdnorm}
  \left\|(A,B,b,d)\right\|_{\mu}:=\left\|\mathrm{vec}(A,B,b,d)\right\|_{\eta},
\end{eqnarray}
where $\left\|\cdot\right\|_{\eta}$ is any kind of vector norm, and with an abuse of notation we take $\mathrm{vec}(A, B, b, d)$  to denote the vector $[\mathrm{vec}(A)^T, \mathrm{vec}(B)^T, b^T, d^T]^T$. One can easily verify that the following pairs of norms satisfy \eqref{def.pdnorm}: $\eta=2$ and $\mu=F$, $\eta=\infty$ and $\mu=\max$, which are typically used in error analysis. Then, the projected condition number of the EILS problem is defined as follows.
\begin{definition}
\label{def.cdEILS}
Considering the map defined by \eqref{def_mapf}, the projected condition number of the EILS problem with respect to $L\in \mathbb{R}^{n\times k}$ with $rank(L)=k$ and the product norm \eqref{def.pdnorm} on the data space is defined as
\begin{align*}
 \kappa_{L\F}(A,B,b,d)&=\lim_{\delta\rightarrow 0}\sup_{\left\|\rho\circ\Delta\right\|_\mu\leq \delta}\frac{\left\|\xi_L\circ\left(\F_L(A+\Delta A, B+\Delta B, b+\Delta b, d+\Delta d)-\F_L(A,B,b,d)\right)\right\|_{\nu}}{\left\|\rho\circ\Delta\right\|_{\mu}},
\end{align*}
where $\F_L(\cdot)$ is defined by \eqref{prjmap}, $\rho\circ\Delta=\left(\Phi\circ\Delta A, \Psi\circ\Delta B, \beta\circ\Delta b, \vartheta\circ\Delta d\right)$, $\xi_L\in\mathbb{R}^{k}$, $\Phi \in {\mathbb{R}^{m \times n}}$, $\Psi \in {\mathbb{R}^{s \times n}}$,  $\beta \in \mathbb{R}^{m}$, and $\vartheta \in {\mathbb{R}^s}$ are  parameters satisfying the requirement in Definition \ref{def.Projcd}, that is, the elements in $(\Phi, \Psi, \beta, \vartheta)$ are nonzero.
\end{definition}

As discussed in \cite{Li14}, the map $\F$ is continuously Fr\'{e}chet differentiable in the neighborhood of $(A,B,b,d)$, and the Fr\'{e}chet derivative of $\F$ at $(A,B,b,d)$ with respect to $(\Delta A, \Delta B, \Delta b, \Delta d)$ is
\begin{align}
\label{DF_EILS}
  D\F(A,B,b,d)\circ(\Delta A, \Delta B, \Delta b, \Delta d) =& M^{-1}B^TN^{-1}(\Delta d- \Delta Bx) - P^TM^{-1}A^TJ (\Delta b- \Delta Ax)\nonumber \\
  & -M^{-1}P(\Delta B^T \lambda+\Delta A^TJr).
\end{align}
With \eqref{eq.kron2}, \eqref{eq.PI1} and \eqref{eq.PI2}, $D\F(A,B,b,d)\circ(\Delta A, \Delta B, \Delta b, \Delta d)$ can also be written as
\begin{align*}
  D\F(A,B,b,d)\circ(\Delta A, \Delta B, \Delta b, \Delta d) =& \left[ x^T\otimes (P^TM^{-1}A^TJ)-(M^{-1}P)\otimes(Jr)^T \right]\mathrm{vec}(\Delta A)\\
    &- \left[(M^{-1}P)\otimes\lambda^T+x^T\otimes(M^{-1}B^TN^{-1})\right]\mathrm{vec}(\Delta B)\\
    &-P^TM^{-1}A^TJ\Delta b  +M^{-1}B^T N^{-1}\Delta d.
\end{align*}
By the Fr\'{e}chet differentiability of $\F$ and $\left\|\rho\circ\Delta\right\|_{\mu}=\|\vec(\Phi, \Psi, \beta, \vartheta)\circ\vec(A,B,b,d)\|_{\eta}$, the explicit expression of the projected condition number of the EILS problem follows from Theorem \ref{Thm.EXFPCD}.
\begin{theorem}
\label{thm.gcd.EILS}
Under Definition \ref{def.cdEILS} and the product norm \eqref{def.pdnorm}, the explicit expression of the projected condition number of the EILS problem \eqref{ILSE} is given by
\begin{eqnarray}
\label{3.1}
  \kappa_{L\F}(A,B,b,d)&=&\left\| \xi_L\circ\left(L^T\mathcal{M}_{\F}\mathrm{diag}(\mathrm{vec}(\Phi, \Psi, \beta, \vartheta)^{\ddagger})\right)\right\|_{\eta,\nu},
\end{eqnarray}
where
\begin{eqnarray}
\label{3.2}
\mathcal{M}_{\F} &=& \begin{bmatrix}
\Gamma, & -\Omega, & -P^TM^{-1}A^TJ, & M^{-1}B^T N^{-1} \\
\end{bmatrix}
\end{eqnarray}
with $\Gamma=x^T\otimes (P^TM^{-1}A^TJ)-(M^{-1}P)\otimes (Jr)^T$, $\Omega= (M^{-1}P)\otimes\lambda^T+x^T\otimes(M^{-1}B^TN^{-1})$, and $\|\cdot\|_{\eta,\nu}$ being the matrix norm induced by the vector norms $\|\cdot\|_\eta$ and $\|\cdot\|_\nu$.
\end{theorem}

Theoretically speaking, Theorem \ref{thm.gcd.EILS} presents the generic form of the projected condition number of the EILS problem. The flexible choice of norms and parameters makes it possible for $\kappa_{L\F}(A,B,b,d)$ to cover the normwise, mixed and componentwise condition numbers as its special cases. But for practical applications we need to specify the parameters and norms. Under the specific setting, we further note that the generic form \eqref{3.1} contains Kronecker product which may make it expensive to compute $\kappa_{L\F}(A,B,b,d)$ with its explicit expression. Thus, under some specific setting of norms, to achieve an efficient computation of the projected condition number makes up the main contents of the following sections.

\begin{theorem}[\bf 2-norm]
\label{thm.closedform}
When $\eta=\nu=2$, $\mu=F$, and the parameters $\Phi$, $\Psi$, $\beta$, $\vartheta$  and $\xi_L$ are positive real numbers, the projected condition number \eqref{3.1} has the following two equivalent and compact forms
\begin{equation}
\label{eq.smpcd1}
\kappa_{2L\F1}(A,B,b,d)  =  \left\|\xi_L^2L^T\mathcal{M}_{pa\F}\mathcal{M}_{pa\F}^T L\right\|^{\frac{1}{2}}_2,
\end{equation}
and
\begin{equation}
\label{eq.smpcd2}
\kappa_{2L\F2}(A,B,b,d)  =  \left\|\xi_LL^T\begin{bmatrix}
                      M^{-1}P\mathcal{Q}, & \frac{\gamma}{\Psi\vartheta}M^{-1}B^TN^{-1}+ \frac{\vartheta}{\Psi\gamma}M^{-1}Px\lambda^T\\
                    \end{bmatrix} \right\|_2,
\end{equation}
where
\begin{align*}
&\mathcal{M}_{pa\F}\mathcal{M}_{pa\F}^T= M^{-1}PSP^TM^{-1} +\frac{1}{\Psi^2}M^{-1}Px\lambda^TN^{-1}BM^{-1}+\frac{1}{\Psi^2}M^{-1}B^TN^{-1}\lambda x^TP^{T}M^{-1} \nonumber \\
   &\qquad\qquad\qquad+\left(\frac{\|x\|^2_2}{\Psi^2}+\frac{1}{\vartheta^2}\right) M^{-1}B^TN^{-2}BM^{-1},\\
&\mathcal{Q}=\begin{bmatrix}
                    \frac{\zeta}{\Phi\beta}A^T-\frac{\beta}{\zeta\Phi}xr^T, & \frac{ \|r\|_2}{\zeta}I_n, &  \frac{\beta\|r\|_2\|x\|_2}{\zeta\Phi}\mathcal{P}_x, & \frac{  \|\lambda\|_2}{\gamma}I_n, & \frac{\vartheta\|\lambda\|_2\|x\|_2}{\gamma\Psi}\mathcal{P}_x \\
                   \end{bmatrix},
\end{align*}
$\zeta^2=\beta^2\|x\|_2^2+\Phi^2$, $\gamma^2=\vartheta^2\|x\|^2_2+\Psi^2$, $S=\left(\frac{\|\lambda\|_2^2}{\Psi^2}+\frac{\|r\|^2_2}{\Phi^2}\right)I_n+\left(\frac{\|x\|_2^2}{\Phi^2}+ \frac{1}{\beta^2}\right)A^TA-\frac{1}{\Phi^2}xr^TA-\frac{1}{\Phi^2}A^Trx^T$, and $\mathcal{P}_x=I_n-\frac{1}{\|x\|^2_2}xx^T$.
\end{theorem}
\begin{proof}
Under the hypothesis of Theorem \ref{thm.closedform}, from Theorem \ref{thm.gcd.EILS} we get
\begin{eqnarray}
\label{eq.unsmpcd}
  \kappa_{2L\F}(A,B,b,d)=\left\|\xi_LL^T\begin{bmatrix}
\frac{1}{\Phi}\Gamma, & -\frac{1}{\Psi}\Omega, & -\frac{1}{\beta} P^TM^{-1}A^TJ, &\frac{1}{\vartheta}M^{-1}B^T N^{-1} \\
\end{bmatrix}\right\|_{2},
\end{eqnarray}
where $\Gamma=x^T\otimes (P^TM^{-1}A^TJ)-(M^{-1}P)\otimes(Jr)^T$ and $\Omega= (M^{-1}P)\otimes\lambda^T+x^T\otimes(M^{-1}B^TN^{-1})$.
By the fact that for any matrix $X\in \mathbb{R}^{m\times n}$, $\left\| X \right\|_2  = \left\| X X^T  \right\|_2^{1/2}$, if we set
\begin{align*}
\mathcal{M}_{pa\F}&=\begin{bmatrix}
\frac{1}{\Phi}\Gamma, & -\frac{1}{\Psi}\Omega, & -\frac{1}{\beta} P^TM^{-1}A^TJ, &\frac{1}{\vartheta}M^{-1}B^T N^{-1} \\
\end{bmatrix},
\end{align*}
then \eqref{eq.unsmpcd} can be written as
\begin{align*}
\kappa_{2L\F}(A,B,b,d)&=\left\|\xi_LL^T\mathcal{M}_{pa\F}\right\|_2\\
&=\left\|\xi_L^2L^T\mathcal{M}_{pa\F}\mathcal{M}_{pa\F}^TL\right\|_2^{\frac{1}{2}}.
\end{align*}
It is easy to verify that $M^{-1}P=P^TM^{-1}$, so with some algebra we get
\begin{eqnarray}
\label{eq.diff1}
  \Gamma \Gamma^T = M^{-1}P(\|x\|_2^2A^TA-xr^TA+\|r\|_2^2I_n-A^Trx^T)P^TM^{-1} ,
\end{eqnarray}
and
\begin{align}
\label{eq.diff2}
  \Omega \Omega^T =& \|\lambda\|_2^2M^{-1}PP^TM^{-1}+M^{-1}Px\lambda^TN^{-1}BM^{-1}+M^{-1}B^TN^{-1}\lambda x^TP^TM^{-1} \nonumber\\
   & +\|x\|_2^2M^{-1}B^TN^{-2}BM^{-1}.
\end{align}
With \eqref{eq.diff1} and \eqref{eq.diff2} we obtain
\begin{align}
\label{eq.closform1}
\mathcal{M}_{pa\F}\mathcal{M}_{pa\F}^T=& M^{-1}PSP^TM^{-1} +\frac{1}{\Psi^2}M^{-1}Px\lambda^TN^{-1}BM^{-1}+\frac{1}{\Psi^2}M^{-1}B^TN^{-1}\lambda x^TP^{T}M^{-1} \nonumber \\
   & +\left(\frac{\|x\|^2_2}{\Psi^2}+\frac{1}{\vartheta^2}\right) M^{-1}B^TN^{-2}BM^{-1},
\end{align}
where $S=\left(\frac{\|\lambda\|_2^2}{\Psi^2}+\frac{\|r\|^2_2}{\Phi^2}\right)I_n+\left(\frac{\|x\|_2^2}{\Phi^2}+ \frac{1}{\beta^2}\right)A^TA-\frac{1}{\Phi^2}xr^TA-\frac{1}{\Phi^2}A^Trx^T$.
Furthermore, \eqref{eq.closform1} can also be written as
\begin{eqnarray*}
   \mathcal{M}_{pa\F}\mathcal{M}_{pa\F}^T = \begin{bmatrix}
                                              M^{-1}P & \frac{1}{\Psi} M^{-1}B^TN^{-1} \\
                                            \end{bmatrix}
   \begin{bmatrix}
     S & \frac{1}{\Psi} x\lambda^T \\
     \frac{1}{\Psi} \lambda x^T & \left(\|x\|^2_2+\frac{\Psi^2}{\vartheta^2}\right)I_n \\
   \end{bmatrix}
   \begin{bmatrix}
     P^TM^{-1} \\
    \frac{1}{\Psi} N^{-1}BM^{-1} \\
   \end{bmatrix}.
\end{eqnarray*}
By block-wise Cholesky factorization, we get
\begin{align}
 \begin{bmatrix}
     S & \frac{1}{\Psi} x\lambda^T \\
     \frac{1}{\Psi} \lambda x^T & \left(\|x\|^2_2+\frac{\Psi^2}{\vartheta^2}\right)I_n \\
   \end{bmatrix} &= \begin{bmatrix}
                       I_n & \frac{\vartheta^2}{\Psi\vartheta^2\|x\|^2_2+\Psi^3}x\lambda^T \\
                       0 & I_n \\
                     \end{bmatrix}
   \begin{bmatrix}
     S-\frac{\vartheta^2\|\lambda\|_2^2}{\vartheta^2\Psi^2\|x\|^2_2+\Psi^4}xx^T & 0 \\
     0 & \left(\|x\|^2_2+\frac{\Psi^2}{\vartheta^2}\right)I_n \\
   \end{bmatrix} \nonumber\\
   & \times\begin{bmatrix}
     I_n & 0 \\
     \frac{\vartheta^2}{\Psi\vartheta^2\|x\|^2_2+\Psi^3}\lambda x^T & I_n \\
   \end{bmatrix}.\label{3.10}
\end{align}
Since
\begin{align*}
S-\frac{\vartheta^2\|\lambda\|_2^2}{\vartheta^2\Psi^2\|x\|^2_2+\Psi^4}xx^T =&\left(\frac{\|x\|_2^2}{\Phi^2}+\frac{1}{\beta^2}\right)A^TA-\frac{1}{\Phi^2}xr^TA- \frac{1}{\Phi^2}A^Trx^T+\left(\frac{\|\lambda\|_2^2}{\Psi^2}+\frac{\|r\|^2_2}{\Phi^2}\right)I_n\\
& -\frac{\vartheta^2\|\lambda\|_2^2}{\vartheta^2\Psi^2\|x\|^2_2+\Psi^4}xx^T,
\end{align*}
we in a similar way obtain that
\small{
\begin{align}
  S-\frac{\vartheta^2\|\lambda\|_2^2}{\vartheta^2\Psi^2\|x\|^2_2+\Psi^4}xx^T &= \begin{bmatrix}
                                            A^T- \frac{\beta^2}{\beta^2\|x\|^2_2+\Phi^2}xr^T & \frac{1}{\Phi} I_n \\
                                          \end{bmatrix}\nonumber\\
  \times& \begin{bmatrix}
   \left(\frac{\|x\|_2^2}{\Phi^2}+\frac{1}{\beta^2}\right)I_n & 0 \\
   0 & \|r\|_2^2\left(I_n-\frac{\beta^2}{\beta^2\|x\|_2^2+\Phi^2}xx^T\right)+ \frac{\Phi^2\|\lambda\|_2^2}{\Psi^2}\left(I_n-\frac{\vartheta^2}{\vartheta^2\|x\|_2^2+\Psi^2}xx^T\right) \\
   \end{bmatrix} \nonumber\\
  \times& \begin{bmatrix}
   A- \frac{\beta^2}{\beta^2\|x\|^2_2+\Phi^2}rx^T\\
   \frac{1}{\Phi} I_n \\
   \end{bmatrix}. \label{3.9}
\end{align}}
From \eqref{3.9}, we note that
\begin{align*}
  I_n-\frac{\beta^2}{\beta^2\|x\|_2^2+\Phi^2}xx^T =& \frac{1}{\beta^2\|x\|^2_2+\Phi^2}\left(\Phi^2I_n+\beta^2\|x\|^2_2 \left(I-\frac{1}{\|x\|_2^2}xx^T\right)\right) \\
   =&\frac{1}{\beta^2\|x\|^2_2+\Phi^2}\begin{bmatrix}
                                   \Phi I_n & \beta\|x\|_2\left(I-\frac{1}{\|x\|_2^2}xx^T\right) \\
                                \end{bmatrix}
\begin{bmatrix}
\Phi I_n \\
 \beta \|x\|_2\left(I-\frac{1}{\|x\|_2^2}xx^T\right) \\
\end{bmatrix}.
\end{align*}
Thus with some algebra we can factorize $S-\frac{\vartheta^2\|\lambda\|_2^2}{\vartheta^2\Psi^2\|x\|^2_2+\Psi^4}xx^T$ as follows
\begin{eqnarray}
\label{3.14}
  S-\frac{\vartheta^2\|\lambda\|_2^2}{\vartheta^2\Psi^2\|x\|^2_2+\Psi^4}xx^T &=& \mathcal{Q}\mathcal{Q}^T,
\end{eqnarray}
where
\begin{align*}
\mathcal{Q}=\begin{bmatrix}
                    \frac{\zeta}{\Phi\beta}A^T-\frac{\beta}{\zeta\Phi}xr^T, & \frac{ \|r\|_2}{\zeta}I_n, &  \frac{\beta\|r\|_2\|x\|_2}{\zeta\Phi}\mathcal{P}_x, & \frac{  \|\lambda\|_2}{\gamma}I_n, & \frac{\vartheta\|\lambda\|_2\|x\|_2}{\gamma\Psi}\mathcal{P}_x \\
                   \end{bmatrix},
\end{align*}
$\zeta^2=\beta^2\|x\|_2^2+\Phi^2$, $\gamma^2=\vartheta^2\|x\|^2_2+\Psi^2$, and $\mathcal{P}_x=I_n-\frac{1}{\|x\|^2_2}xx^T$.
Just substituting \eqref{3.14} into \eqref{3.10} and repeating the process of deriving \eqref{3.14}, we have
\begin{eqnarray}
\label{eq.fmeils}
L^T\mathcal{M}_{pa\F}\mathcal{M}_{pa\F}^TL &=& L^T\mathcal{K}\mathcal{K}^TL,
\end{eqnarray}
where $ \mathcal{K}=\begin{bmatrix}
                      M^{-1}P\mathcal{Q}, & \frac{\gamma}{\Psi\vartheta}M^{-1}B^TN^{-1}+ \frac{\vartheta}{\Psi\gamma}M^{-1}Px\lambda^T\\
                    \end{bmatrix}.$ By \eqref{eq.closform1} and \eqref{eq.fmeils}, we complete the proof.
\end{proof}
\begin{remark}\label{rmk3}\rm
The explicit expression of the 2-norm projected condition number of the EILS problem  $\kappa_{2L\F}(A,B,b,d)$ ( or  \eqref{eq.unsmpcd}) is firstly established. But the main contribution of Theorem~\ref{thm.closedform} is that it gives two equivalent but more compact forms of $\kappa_{2L\F}(A,B,b,d)$. Note that both \eqref{eq.smpcd1} and \eqref{eq.smpcd2} eliminate the Kronecker product, and the orders of matrices in \eqref{eq.unsmpcd}, \eqref{eq.smpcd1} and \eqref{eq.smpcd2} are $k\times (n+1)(m+s)$, $k\times k$, and $k\times (4n+m+s)$, respectively. Let $k=n$. When $m$, $n$ and $s$ are comparable and large and the exact value of the projected condition number is needed, the explicit computation of condition number with \eqref{eq.unsmpcd} becomes impossible due to its large order. The superiority of \eqref{eq.smpcd1} and \eqref{eq.smpcd2} becomes apparent since they need much less storage space and can be efficiently computed. One can easily find that \eqref{eq.smpcd1} is even more compact than \eqref{eq.smpcd2}.  A numerical comparison of these three equivalent forms of the 2-norm projected condition number will be given in Section \ref{sec.5}.
\end{remark}

Now we consider the projected condition number with $\eta=\nu=\infty$ and $\mu=\max$ for the EILS problem, from which the projected mixed and componentwise condition numbers follow directly.
\begin{theorem}[\bf $\infty$-norm]
\label{thm.infnorm}
When $\eta=\nu=\infty$ and $\mu=\max$, the projected condition number of the EILS problem \eqref{ILSE} is given by
\begin{eqnarray}\label{eq.m&c}
\kappa_{{\infty}L\F}(A,B,b,d)&=&\left\|\xi_L\circ\left(L^T\mathcal{M}_{\F}\mathrm{diag}(\mathrm{vec}(\Phi, \Psi, \beta, \vartheta)^\ddag)\right)\right\|_{\infty}\nonumber\\
 &=&\left\| \left|\xi_L\right|\circ\left(\left|L^T\mathcal{M}_{\F}\right|  \left|\mathrm{vec}(\Phi, \Psi, \beta, \vartheta)^\ddag\right|\right) \right\|_{\infty},
\end{eqnarray}
where $\mathcal{M}_{\F}$ is defined by \eqref{3.2}. Particularly, if $\Phi=A^{\ddag}$, $\Psi=B^{\ddag}$, $\beta=b^{\ddag}$, $\vartheta=d^{\ddag}$, and $\xi_L=1/\|L^Tx\|_{\infty}$ and $\left(L^Tx\right)^{\ddag}$ sequentially, then the projected mixed and componentwise condition numbers of the EILS problem follow correspondingly
\begin{eqnarray}
  \kappa_{\infty{L\F}}^{m}(A,B,b,d) &=&\frac{\left\|{\left|L^T\mathcal{M}_{\F}\right|  \left|\mathrm{vec}\left(A^{\ddag}, B^{\ddag}, b^{\ddag}, d^{\ddag}\right)^{\ddag}\right|} \right\|_{\infty}}{\left\|L^Tx\right\|_{\infty}},\label{eq.mixcd} \\
  \kappa_{\infty{L\F}}^{c}(A,B,b,d) &=& \left\|\frac{\left|L^T\mathcal{M}_{\F}\right| \left|\mathrm{vec}\left(A^{\ddag}, B^{\ddag}, b^{\ddag}, d^{\ddag}\right)^{\ddag}\right|}{\left|L^Tx\right|}\right\|_{\infty}.\label{eq.componentcd}
\end{eqnarray}
\end{theorem}
\begin{proof} Letting $\eta=\nu=\infty$ in \eqref{3.1} gives the first part of \eqref{eq.m&c}. The second part of \eqref{eq.m&c} can be obtained by considering the proof of Lemma 2 in \cite{Cucker07} and \eqref{eq.cdivis}, which is straightforward, so we omit it here.
\end{proof}

\begin{remark}
\label{rmk3.6}
\rm
Although Theorem \ref{thm.infnorm} gives the explicit expressions of the projected mixed and componentwise condition numbers, it is still difficult to compute the exact value of these condition numbers due to the Kronecker product. Another important issue is that we can not find compact forms of the projected mixed and componentwise condition numbers similar to Theorem \ref{thm.closedform}. If we look into the explicit expressions of the projected mixed and componentwise condition numbers, we note that the main computational difficulty lies in the following term
\begin{align*}
\left|L^T\Gamma\right|\left|\mathrm{vec}(A^{\ddag})^{\ddag}\right|=\left|L^T\left(x^T\otimes (P^TM^{-1}A^TJ)-(M^{-1}P)\otimes(Jr)^T\right)\right|\left|\mathrm{vec}(A^{\ddag})^{\ddag}\right|.
\end{align*}
By absolute value inequality and \eqref{eq.kron2}, $\left|L^T\Gamma\right|\left|\mathrm{vec}(A^{\ddag})^{\ddag}\right|$ may be bounded by
\begin{align*}
\left|L^T\Gamma\right|\left|\mathrm{vec}(A^{\ddag})^{\ddag}\right|\leq |L^TP^TM^{-1}A^TJ|\left|(A^{\ddag})^{\ddag}\right||x|+|L^TM^{-1}P|\left|((A^T)^{\ddag})^{\ddag}\right||Jr|.
\end{align*}
With a similar treatment to $\left|L^T\Omega\right|\mathrm{vec}(|B|)$, if we set
\begin{align*}
\mathcal{M}_{mc}^{Ubd}=&\left|L^TM^{-1}B^TN^{-1}\right|\left(\left|(d^{\ddag})^{\ddag}\right|+ \left|(B^{\ddag})^{\ddag}\right||x|\right) +\left|L^TP^TM^{-1}A^TJ\right|\left(\left|(b^{\ddag})^{\ddag}\right|+ \left|(A^{\ddag})^{\ddag}\right||x|\right)\\
&+\left|L^TM^{-1}P\right|\left(\left|((B^T)^{\ddag})^{\ddag}\right| |\lambda|+\left|((A^T)^{\ddag})^{\ddag}\right||Jr|\right),
\end{align*}
 then the projected mixed and componentwise condition numbers can be bounded by
\begin{eqnarray*}
&\kappa^{mU}_{\infty{L\F}}(A,B,b,d)=
 \frac{\left\|\mathcal{M}_{mc}^{Ubd} \right\|_{\infty}}{\|L^Tx\|_{\infty}},\textrm{ and }
\kappa^{cU}_{\infty{L\F}}(A,B,b,d)=
 \left\|\frac{\mathcal{M}_{mc}^{Ubd}}{|L^Tx|} \right\|_{\infty}.
\end{eqnarray*}
In our numerical experiments, we will show that these upper bounds are not only very tight but also can be efficiently computed. Moreover, we need to point out that Li et al. \cite{Li14} considered the mixed and componentwise condition numbers of the ILS problem with LS constraints with the definition given in \cite{Gohb93}, which makes their condition numbers can be infinite when $x$ contains zero element. The superiority of our projected mixed and componentwise condition numbers is that it always give finite values and can be used to reflect the conditioning of certain elements in the solution.
\end{remark}

\section{Some specific LS problems}
\label{sec.4}
Based on the results on the projected condition numbers of the EILS problem, we present some interesting new findings on the condition number theory of some specific LS problems.
\subsection{The ILS problem}
If we remove the equality constraints $Bx=d$, then the ILS problem follows from \eqref{ILSE}
\begin{equation}\label{ILS}
 \mathop {\min }\limits_{x \in\mathbb{R} {^n}} {(b - Ax)^T}J(b - Ax).
\end{equation}
Note that the condition number theory of ILS problem has been studied \cite{Li14,LiW16aX,DiaoZ16}. In this paper we only present a new result on its 2-norm projected condition number.

Let $B=0$ and $d=0$. Then the Fr\'{e}chet derivative of $\F$ at $(A,b)$ with respect to $(\Delta A, \Delta b)$ can be easily derived from \eqref{DF_EILS} and
\begin{align}
\label{DF_ILS}
  D\F(A,b)\circ(\Delta A, \Delta b) =&  M^{-1}A^TJ (\Delta b- \Delta Ax)+M^{-1}\Delta A^TJr.
\end{align}
From Theorems \ref{thm.closedform} and \ref{thm.infnorm}, the 2-norm projected condition number of the ILS problem can be easily obtained.
\begin{theorem}
\label{thm.closedILS}
Let $B=0$ and $d=0$. When $\eta=\nu=2$, $\mu=F$, and the parameters $\Phi$, $\Psi$, $\beta$, $\vartheta$ and $\xi_L$ are positive real numbers, the 2-norm projected condition number of the ILS problem is given by
\begin{eqnarray*}
\kappa_{2L\F_{ils}} (A,b)=\left\|\xi_LL^T\left[\frac{1}{\Phi} \Gamma_J,\frac{1}{\beta} { M^{ - 1} A^T J}\right] \right\|_2,
\end{eqnarray*}
where $\Gamma_J  = M^{ - 1}\otimes (Jr)^T - x^T  \otimes ( M^{ - 1} A^T J).$
Then, we also have the following two equivalent expressions
\begin{align}\label{cdILS1}
\kappa_{2L\F_{ils1}} (A,b)&=  \left\|\xi^2_LL^TM^{ - 1} \left( { \frac{\zeta^2}{\Phi^2\beta^2} A^T A+\frac{\left\| r \right\|_2^2}{\Phi^2} I_n  - \frac{1}{\Phi^2}(xr^T A + A^T rx^T )} \right)M^{ - 1}L\right\|^{\frac{1}{2}}_2
\end{align}
and
\begin{eqnarray}\label{cdILS2}
\kappa_{2L\F_{ils2}} (A,b)=\left\|\xi_LL^TM^{-1}\begin{bmatrix}
                    \frac{\zeta}{\Phi\beta}A^T-\frac{\beta}{\zeta\Phi}xr^T, & \frac{ \|r\|_2}{\zeta}I_n, &  \frac{\beta\|r\|_2\|x\|_2}{\zeta\Phi}\mathcal{P}_x \\
                   \end{bmatrix}\right\|_2,
\end{eqnarray}
where $\zeta^2=\beta^2\|x\|_2^2+\Phi^2$ and $\mathcal{P}_x=I_n-\frac{1}{\|x\|^2_2}xx^T$.
\end{theorem}
\begin{proof}
Since $B=0$, we get
\begin{align*}
\lambda&=0,
 P=-I_n, \textrm{ and }
 \Omega=0.
\end{align*}
Thus, $\kappa_{2L\F_{ils}} (A,b)$, $\kappa_{2L\F_{ils1}} (A,b)$ and $\kappa_{2L\F_{ils2}} (A,b)$ follow from \eqref{eq.unsmpcd}, \eqref{eq.smpcd1} and \eqref{eq.smpcd2}, respectively.
\end{proof}
\begin{remark}
\rm
The projected condition number of the ILS problem has been studied by Li and Wang \cite{LiW16aX} with the notation partial condition number, and \eqref{cdILS1} has been given in \cite[Theorem~3.2, Equation (3.7)]{LiW16aX}. Like \eqref{cdILS2}, a compact form \cite[Theorem~3.2, Equation (3.8)]{LiW16aX} was also given. But, we need to point out that \eqref{cdILS2} is still different from their Eqution (3.8). The orders of matrices in \eqref{cdILS2} and (3.8) in \cite[Theorem~3.2]{LiW16aX} are
 $k\times (2n+m)$ and $k\times (2m+n)$ respectively, so our \eqref{cdILS2} is still more compact than (3.8) with the assumption that $m>n$. The explicit expressions of the projected mixed and componentwise condition numbers and its upper bounds have been given in~\cite{LiW16aX}. Diao and Zhou \cite{DiaoZ16} used the dual techniques to recover the explicit expressions of mixed and componentwise condition numbers of the ILS problem. Thus, considering the relationship between the EILS and the ILS problems, we may say that the results given in \cite{Li14,LiW16aX,DiaoZ16} can be treated as special cases of our work. Moreover, based on the relationship between the ILS and the TLS problems \cite{Cha}, Li and Wang \cite{LiW16aX} also established the condition number of the TLS problem but they did not give the compact forms, which were later given in \cite{WangL17TLS}. More results on the condition number theory of the TLS problem can be found in \cite{Bab,JiaL13,LiJ11,ZhMW17}.
\end{remark}

\subsection{The WLS problem}

If we substitute the matrix $J$ in \eqref{ILS} with a positive definite matrix $W\in \mathbb{R}^{m\times m}$, then the weighted least squares (WLS) problem \cite[Ch.~4]{Bjo96} follows as
\begin{equation}\label{WLS}
 \mathop {\min }\limits_{x \in\mathbb{R} {^n}} {(b - Ax)^T}W(b - Ax).
\end{equation}
The normwise condition number of the WLS problem has been studied in \cite{WeiW03, WangZ09, YangW16} with a slightly different setting that the solution $x$ to \eqref{WLS} is minimum in weighted $N$-norm \cite{WangW04}, $\|x\|_N=\sqrt{x^TNx}$ with $N\in\mathbb{R}^{n\times n}$ being positive definite, and the derivation in \cite{WeiW03, WangZ09} heavily relies on the weighted singular value decomposition (WSVD) \cite{WangW04}.

Since the substitution of $J$ does not change the Fr\'{e}chet differentiability of $\F$, the Fr\'{e}chet derivative of $\F$ at $(A,b)$ with respect to $(\Delta A, \Delta b)$ is given by
\begin{align}
\label{DF_WLS}
  D\F(A,b)\circ(\Delta A, \Delta b) =&  (A^TWA)^{-1}A^TW (\Delta b- \Delta Ax)+(A^TWA)^{-1}\Delta A^TWr.
\end{align}
Let $B=0$ and $d=0$. From Theorem \ref{thm.gcd.EILS}, the generic form of the projected condition number for the WLS problem can be easily obtained.
\begin{theorem}
\label{thm.gcd.WLS}
Let $B=0$ and $d=0$. Substituting the matrix $J$ in \eqref{ILSE} with the matrix $W$ in \eqref{WLS} and from Theorem \ref{thm.gcd.EILS}, the explicit expression of the projected condition number for the WLS problem is given by
\begin{eqnarray}
\label{eq.gcdwls}
  \kappa_{L\F_{wls}}(A,b)&=&\left\|\xi_L\circ\left(L^T\begin{bmatrix}
                                                        \Gamma_W, & (A^TWA)^{-1}A^TW  \\
                                                      \end{bmatrix}
  \mathrm{diag}\left(\mathrm{vec}(\Phi, \beta)^{\ddag}\right)\right)\right\|_{\eta,\nu},
\end{eqnarray}
where $\Gamma_W=(A^TWA)^{-1}\otimes (Wr)^T-x^T\otimes (A^TWA)^{-1}A^TW$ and $\|\cdot\|_{\eta,\nu}$ is the matrix norm induced by the vector norms $\|\cdot\|_\eta$ and $\|\cdot\|_\nu$.
\end{theorem}
\begin{proof}
The proof of Theorem \ref{thm.gcd.WLS} is similar to that of Theorem \ref{thm.closedILS}, which is omitted here.
\end{proof}

It should be noted that the WLS problem \eqref{WLS} actually gives an weighted norm to measure the residual in the data space. For practical applications, we introduce the following weighted product norm on the data space
\begin{equation}\label{def.wpdnorm}
  \|(A,b)\|_{WF}:=\left\|\vec(A,b)\right\|_{W_{\otimes I}},
\end{equation}
where $W_{\otimes I}=\begin{bmatrix}
                          I_n\otimes W &   \\
                            & W \\
                        \end{bmatrix}$ and $\left\|\vec(A,b)\right\|_{W_{\otimes I}}=\sqrt{\vec(A,b)^TW_{\otimes I}\vec(A,b)}$.
With the weighted product norm \eqref{def.wpdnorm}, we present the compact forms of the 2-norm projected condition number of the WLS problem in the following theorem.
\begin{theorem}
\label{thm.closedWLS}
With the product norm defined by \eqref{def.wpdnorm}, if we set $\nu=2$ and the parameters $\Phi$, $\beta$ and $\xi_L$ are positive real numbers, then $\kappa_{L\F_{wls}}(A,b)$ can be further simplified into the following two equivalent forms
\begin{equation}\label{eq.cd1wls}
   \kappa_{2L\F_{wls1}}(A,b)=\left\|\xi_L^2L^T\left(\frac{\|r\|_{W}^2}{\Phi^2}(A^TWA)^{-2} +\left(\frac{\|x\|_2}{\Phi^2}+\frac{1}{\beta^2}\right)(A^TWA)^{-1}\right)L\right\|_2^{\frac{1}{2}},
\end{equation}
and
\begin{equation}\label{eq.cd2wls}
   \kappa_{2L\F_{wls2}}(A,b)=\left\|\xi_LL^T(A^TWA)^{-1}\begin{bmatrix}
                                                        \frac{ \|r\|_W}{\Phi}I_n, & \frac{\zeta}{\Phi\beta} A^T W^{\frac{1}{2}}\\
                                                      \end{bmatrix}
   \right\|_2,
\end{equation}
where $\|r\|_W=(r^TWr)^{1/2}$, $\zeta^2=\beta^2\|x\|_2^2+\Phi^2$.
\end{theorem}
\begin{proof}
Under the hypothesis of Theorem \ref{thm.closedWLS}, we get
\begin{eqnarray}
\label{eq.cdwls2}
  \kappa_{2L\F_{wls}}(A,b)&=&\left\|\xi_L\left(L^T\begin{bmatrix}
                                                        \frac{1}{\Phi}\Gamma_W, & \frac{1}{\beta}(A^TWA)^{-1}A^TW  \\
                                                      \end{bmatrix}\right)\right\|_{W_{\otimes I},2}.
\end{eqnarray}
From \cite{WangW04} and \cite{WangY17}, $\kappa_{2L\F_{wls}}(A,b)$ can also be written as
\begin{eqnarray}
\label{eq.cdwls22}
  \kappa_{2L\F_{wls}}(A,b)&=&\left\|\xi_L
  \left(L^T\begin{bmatrix}
  \frac{1}{\Phi}\Gamma_W, & \frac{1}{\beta}(A^TWA)^{-1}A^TW  \\
  \end{bmatrix}
  \begin{bmatrix}
  I_n\otimes W^{-\frac{1}{2}} &   \\
   & W^{-\frac{1}{2}} \\
   \end{bmatrix}\right)\right\|_{2},
\end{eqnarray}
where $W^{-\frac{1}{2}}$ is the square root of $W^{-1}$. Thus, let
\begin{equation}\label{eq.Mwls}
 \mathcal{M}_{pa\F_{wls}}=\begin{bmatrix}
  \frac{1}{\Phi}\Gamma_W, & \frac{1}{\beta}(A^TWA)^{-1}A^TW  \\
  \end{bmatrix}
  \begin{bmatrix}
  I_n\otimes W^{-\frac{1}{2}} &   \\
   & W^{-\frac{1}{2}} \\
   \end{bmatrix},
\end{equation}
then
\begin{align*}
\mathcal{M}_{pa\F_{wls}}\mathcal{M}_{pa\F_{wls}}^T&=\frac{\|r\|_{W}^2}{\Phi^2}(A^TWA)^{-2} +\left(\frac{\|x\|_2}{\Phi^2}+\frac{1}{\beta^2}\right)(A^TWA)^{-1}\nonumber\\
&-\frac{1}{\Phi^2}\left((A^TWA)^{-1}xr^TWA(A^TWA)^{-1}+(A^TWA)^{-1}A^TWrx^T(A^TWA)^{-1}\right).
\end{align*}
Since
\begin{align}
\label{eq.wls0}
 (A^TWA)^{-1}A^TWr=(A^TWA)^{-1}A^TWb-x=0,
\end{align}
we get
\begin{align}
\label{eq.wlsMMT1}
\mathcal{M}_{pa\F_{wls}}\mathcal{M}_{pa\F_{wls}}^T&=\frac{\|r\|_{W}^2}{\Phi^2}(A^TWA)^{-2} +\left(\frac{\|x\|_2}{\Phi^2}+\frac{1}{\beta^2}\right)(A^TWA)^{-1}.
\end{align}
It easy to check that
\begin{align}
\label{eq.wlsMMT2}
\mathcal{M}_{pa\F_{wls}}\mathcal{M}_{pa\F_{wls}}^T=(A^TWA)^{-1}\begin{bmatrix}
                                                        \frac{ \|r\|_W}{\Phi}I_n, & \frac{\zeta}{\Phi\beta} A^T W^{\frac{1}{2}}\\
                                                      \end{bmatrix}
                                                      \begin{bmatrix}
                                                        \frac{ \|r\|_W}{\Phi}I_n\\
                                                         \frac{\zeta}{\Phi\beta} W^{\frac{1}{2}} A \\
                                                      \end{bmatrix}(A^TWA)^{-1}.
\end{align}
Therefore, $\kappa_{2L\F_{wls1}}(A,b)$ and $\kappa_{2L\F_{wls2}}(A,b)$ can be easily obtained with \eqref{eq.wlsMMT1} and \eqref{eq.wlsMMT2}.
\end{proof}
\begin{remark}
\label{rmk4.2}
\rm
Different from \cite{WeiW03, WangZ09}, our method does not rely on the WSVD in establishing the explicit expression of the projected condition number for the WLS problem. In addition, when $L=I_n$, \eqref{eq.cd1wls} has been given in \cite{YangW16}, but as far as we know \eqref{eq.cd2wls} is a new result. Moreover, if we compare Theorem \ref{thm.closedWLS} with Theorem~\ref{thm.closedILS}, we note that \eqref{cdILS1} and \eqref{cdILS2} are more complicated than \eqref{eq.cd1wls} and \eqref{eq.cd2wls}, respectively. The reason is that $W$ is positive definite and can be factorized as $W=W^{\frac{1}{2}}W^{\frac{1}{2}}$, so with the product norm \eqref{def.wpdnorm} we get \eqref{eq.wls0}
in deriving \eqref{eq.cd1wls}, which does not hold in establishing \eqref{cdILS1}. This can be treated as an intrinsic distinction between the ILS problem and the WLS problem. Note that when $W$ reduces to $I_m$, the LS problem follows from \eqref{WLS}. The explicit expressions of the projected condition numbers of the LS problem have been given in \cite{LiW16aX} with the notation partial condition number, so we will not discuss the condition number theory of the LS problem in the present paper.
\end{remark}

By changing the norms and parameters, we can also get the projected mixed and componentwise condition numbers and its upper bounds of the WLS problem from Theorem \ref{thm.gcd.WLS}. The results are summerized in the following theorem.
\begin{theorem}
\label{thm.WLS}
When $\eta=\nu=\infty$ and $\mu=\max$, the projected condition number of the WLS problem is given by
\begin{equation*}\label{}
  \kappa_{\infty L\F wls}(A,b)=\left\| \left|\xi_L\right|\circ\left(\left|L^T\begin{bmatrix}
  \Gamma_W, & (A^TWA)^{-1}A^TW  \\
  \end{bmatrix}\right|  \left|\mathrm{vec}(\Phi, \beta)^{\ddag}\right|\right) \right\|_{\infty}.
\end{equation*}
In particular, if we set $\Phi=A^{\ddag}$, $\beta=b^{\ddag}$, and $\xi_L=1/\|L^Tx\|_{\infty}$ and $(L^Tx)^{\ddag}$ in turn, then the projected mixed and componentwise condition numbers are given as follows
\begin{align*}
  \kappa_{\infty L\F wls}^{m}(A,b)=& \frac{\left\|\left|L^T\begin{bmatrix}
  \Gamma_W, & (A^TWA)^{-1}A^TW  \\
  \end{bmatrix}\right| \left|\mathrm{vec}(A^{\ddag}, b^{\ddag})^{\ddag}\right|\right\|_{\infty}}{\|L^Tx\|_{\infty}}, \\
  \kappa_{\infty L\F wls}^{c}(A,b)=& \left\| \frac{\left|L^T \begin{bmatrix}
  \Gamma_W, & (A^TWA)^{-1}A^TW  \\
  \end{bmatrix}\right|  \left|\mathrm{vec}(A^{\ddag}, b^{\ddag})^{\ddag}\right|}{\left|L^Tx\right|}\right\|_{\infty},
\end{align*}
where $\Gamma_W=(A^TWA)^{-1}\otimes (Wr)^T-x^T\otimes (A^TWA)^{-1}A^TW$. Similarly, we can also get its computable upper bounds as follows
\begin{align*}
  &\kappa_{\infty L\F wls}^{mU}(A,b)=  \frac{\left\|\M_{Wmc}^{Ubd}\right\|_{\infty}}{\|L^Tx\|_{\infty}} \quad \mathrm{ and }\quad
  \kappa_{\infty L\F wls}^{cU}(A,b)=  \left\| \frac{\M_{Wmc}^{Ubd} }{\left|L^Tx\right|}\right\|_{\infty},
\end{align*}
\end{theorem}
where
\begin{align*}
  \M_{Wmc}^{Ubd} & = |L^T(A^TWA)^{-1}|\left|((A^T)^{\ddag})^{\ddag}\right||Wr|+ |L^T(A^TWA)^{-1}A^TW|\left(\left|(A^{\ddag})^{\ddag}\right||x|+\left|(b^{\ddag})^{\ddag}\right|\right).
\end{align*}
The mixed and componentwise condition numbers of the WLS problem have been studied in \cite{LiS09}. Due to their definition, the condition numbers given in \cite{LiS09} can be infinite and can not be used to give the conditioning of certain elements in the solution. Moreover, it can be easily checked that when $L=I_n$, our projected mixed and componentwise condition numbers and its upper bounds cover the results given in \cite{LiS09} as special cases.

\subsection{The ELS problem}
When the signature matrix $J$ in \eqref{ILSE} is replaced by $I_m$, the ELS problem follows from \eqref{ILSE}
\begin{equation}\label{ELS}
  {\rm ELS:}\quad \mathop {\min }\limits_{x \in\mathbb{R} {^n}} \|b - Ax\|_2^2 \quad \textrm{ subject to } Bx = d.
\end{equation}
We can check that the assumption \eqref{asump1} degenerates to the following condition which guarantees the ELS problem \eqref{ELS} to have a unique solution and can be found in
\cite{Bjo96},
\begin{equation}\label{asumpELS}
{\rm rank}(B)=s,\quad {\rm null}(A)\cap{\rm null}(B)=\{0\}.
\end{equation}
Analogous to \eqref{agmt.eq},  the solution of \eqref{ELS} also satisfies the following augmented system
\begin{align}\label{agmELS}
    \begin{bmatrix}
     0 & 0 & B \\
     0 & I & A \\
     B^T & A^T & 0 \\
   \end{bmatrix}\begin{bmatrix}
                  \lambda \\
                  r \\
                  x \\
                \end{bmatrix}
   &= \begin{bmatrix}
          d \\
          b \\
          0 \\
        \end{bmatrix},
\end{align}
where $\lambda=-(BB^T)^{-1}BA^Tr$ is the vector of Lagrange multipliers.

In the following discussion, we confine ourself to the ELS problem which is only different from the EILS problem by the signature matrix $J$. To avoid introducing more notation, we adopt the symbols used in Section~\ref{sec.3}, which should cause no confusion. Under the assumption \eqref{asumpELS} and from \cite{Eld80}, we get that the coefficient matrix of \eqref{agmELS} is invertible, and its inverse is
\begin{equation*}\label{}
\begin{bmatrix}
  N^{-1} & -N^{-1}BM^{-1}A^T & N^{-1}BM^{-1} \\
  -AM^{-1}B^TN^{-1} & I+AP^TM^{-1}A^T & -AM^{-1}P \\
  M^{-1}B^TN^{-1} & -P^TM^{-1}A^T & M^{-1}P \\
\end{bmatrix},
\end{equation*}
where $M=A^TA$, $N=BM^{-1}B^T$ and $P=B^TN^{-1}BM^{-1}-I$. So the solution to \eqref{ELS} is
\begin{equation}\label{eq.slELS}
  x=M^{-1}B^TN^{-1}d-P^TM^{-1}A^Tb.
\end{equation}
We should note that \eqref{eq.slELS} seems to be different from the following widely used form (cf. \cite{Bjo96,Eld80})
\begin{equation}\label{eq.slELS2}
x=B^{\dagger}_{A}d+(A(I_n-B^{\dagger}B))^{\dagger}b,
\end{equation}
where $B^{\dagger}_{A}=(I_n-(A(I_n-B^{\dagger}B))^{\dagger}A)B^{\dagger}$, and $B^{\dagger}$ is the Moore-Penrose inverse of $B$. With the definition of Moore-Penrose inverse (cf. \cite{WangW04}), it can be checked that \eqref{eq.slELS} and \eqref{eq.slELS2} are equivalent. Thus in the following discussion we will not use the Moore-Penrose inverse.

The substitution of $J$ with $I_m$ does not change the Fr\'{e}chet differentiability of $\F$ defined by \eqref{def_mapf}, the generic expression of the projected condition number for the ELS problem can be obtained from Theorem \ref{thm.gcd.EILS}.

\begin{theorem}\label{Thm.gcdELS}Let $J=I_m$. Then, according to Theorem \ref{thm.gcd.EILS} the projected condition number of the ELS problem \eqref{ELS} is given by
\begin{eqnarray}
\label{pcdELS}
  \kappa_{L\F els}(A,B,b,d)&=&\left\| \xi_L\circ \left(L^T\mathcal{M}_{\F}\mathrm{diag}\left(\mathrm{vec}(\Phi, \Psi, \beta, \vartheta)^{\ddagger}\right)\right)\right\|_{\eta,\nu},
\end{eqnarray}
where
\begin{eqnarray}
\label{3.22}
\mathcal{M}_{\F} &=& \begin{bmatrix}
\Gamma, & -\Omega, & -P^TM^{-1}A^T, & M^{-1}B^T N^{-1} \\
\end{bmatrix}
\end{eqnarray}
with $\Gamma=x^T\otimes (P^TM^{-1}A^T)-(M^{-1}P)\otimes r^T$, $\Omega= (M^{-1}P)\otimes\lambda^T+x^T\otimes(M^{-1}B^TN^{-1})$, and $\|\cdot\|_{\eta,\nu}$ being the matrix norm induced by the vector norms $\|\cdot\|_\eta$ and $\|\cdot\|_\nu$.
\end{theorem}

Parallel to Theorem \ref{thm.closedform}, we also get two simplified equivalent forms of the 2-norm projected condition number for the ELS problem.
\begin{theorem}
\label{thm.clfmELS}
When $\eta=\nu=2$, $\mu=F$, and the parameters $\Phi$, $\Psi$, $\beta$, $\vartheta$  and $\xi_L$ are positive real numbers, the projected condition number \eqref{pcdELS} has the following two equivalent forms
\begin{equation}
\label{eq.smpcdELS1}
\kappa_{2L\F els1}(A,B,b,d)  =  \frac{\left\|L^T\mathcal{M}_{pa\F}\mathcal{M}_{pa\F}^T L\right\|^{\frac{1}{2}}_2}{\xi_L},
\end{equation}
and
\begin{equation}
\label{eq.smpcdELS2}
\kappa_{2L\F els2}(A,B,b,d)  =  \frac{\left\|L^T\begin{bmatrix}
                      M^{-1}P\mathcal{Q}, & \frac{\gamma}{\Psi\vartheta}M^{-1}B^TN^{-1}+ \frac{\vartheta}{\Psi\gamma}M^{-1}Px\lambda^T\\
                    \end{bmatrix} \right\|_2}{\xi_L},
\end{equation}
where
\begin{align*}
&\mathcal{M}_{pa\F}\mathcal{M}_{pa\F}^T= M^{-1}PSP^TM^{-1} +\frac{1}{\Psi^2}M^{-1}Px\lambda^TN^{-1}BM^{-1}+\frac{1}{\Psi^2}M^{-1}B^TN^{-1}\lambda x^TP^{T}M^{-1} \nonumber \\
   &\qquad\qquad\qquad+\frac{\gamma^2}{\Psi^2\vartheta^2}M^{-1}B^TN^{-2}BM^{-1},\\
&\mathcal{Q}=\begin{bmatrix}
                    \frac{\zeta}{\Phi\beta}A^T, & \frac{ \|r\|_2}{\Phi}I_n, & \frac{  \|\lambda\|_2}{\gamma}I_n, & \frac{\vartheta\|\lambda\|_2\|x\|_2}{\gamma\Psi}\mathcal{P}_x \\
                   \end{bmatrix},
\end{align*}
$\zeta^2=\beta^2\|x\|_2^2+\Phi^2$, $\gamma^2=\vartheta^2\|x\|^2_2+\Psi^2$, $S=\left( \frac{\|\lambda\|_2^2}{\Psi^2}+\frac{\|r\|^2_2}{\Phi^2}\right)I_n+\frac{\zeta^2}{\beta^2\Phi^2} A^TA$,
and $\mathcal{P}_x=I_n-\frac{1}{\|x\|^2_2}xx^T$.
\end{theorem}
\begin{proof}
The proof of Theorem \ref{thm.clfmELS} is similar to that of Theorem \ref{thm.closedform}. The only difference is that when $J=I_m$, from \eqref{eq.closform1} we get the following term
\begin{align}
\label{eq.eqltyELS}
P^TM^{-1}A^Tr=P^TM^{-1}A^T(b-Ax)=P^TM^{-1}A^Tb-M^{-1}B^TN^{-1}d+x=0,
\end{align}
which does not hold for the EILS problem. Since the rest of the proof is similar to that of Theorem \ref{thm.closedform}, we omit it here.
\end{proof}

\begin{remark}
\rm
The equality \eqref{eq.eqltyELS} makes \eqref{eq.smpcdELS1} and \eqref{eq.smpcdELS2} much simpler than \eqref{eq.smpcd1} and \eqref{eq.smpcd2}, respectively. The same phenomenon is also discussed in Remark \ref{rmk4.2}. Moreover, we need to point out that the condition numbers for the ELS problem have been extensively studied in \cite{LiW17} in a concrete framework.  Li and Wang \cite{LiW17} presented the explicit expressions of the projected normwise, mixed and componentwise condition numbers for the ELS problem, but \eqref{eq.smpcdELS2} was not obtained there.
\end{remark}

\section{Numerical experiments}\label{sec.5}

In this part, the random EILS problems are used to illustrate the utility of the proposed condition numbers. We construct the random EILS problem as follows. The coefficient matrices $A$ and $B$ are given by
\begin{eqnarray*}
  A=HD\begin{bmatrix}
                    Q_2 \\
                    Q_1 \\
                  \end{bmatrix}
   &\quad \mathrm{and}\quad & B=\begin{bmatrix}
              K & 0 \\
            \end{bmatrix}\begin{bmatrix}
                    Q_1 \\
                    Q_2 \\
                  \end{bmatrix},
\end{eqnarray*}
where $H$ is $J$-orthogonal, i.e., $H^TJH=J$ , and  generated via the method given in \cite{High03}. $D\in\mathbb{R}^{m\times n}$ is diagonal matrix with decreasing diagonal values geometrically distributed between $\kappa_A$ and $1$.
$Q=\begin{bmatrix}
Q_1 \\
Q_2 \\
\end{bmatrix}$ is a random orthogonal matrix generated by the Matlab command \texttt{gallery(¡¯qmult¡¯,n)}.  $K$ is a lower triangular matrix and generated by QR factorization of random matrix with specified condition number and preassigned singular value distribution. It is easy to check that the condition number of $H$, $\kappa(H)\geq 1$, and $\kappa(XY)\leq \kappa(X)\kappa(Y)$ holds for matrices $X$ and $Y$. Therefore, we can keep $B$ having a specific condition number, and $\kappa(A)$ in certain level of magnitude. Then, we set the solution $x\in \mathbb{R}^{n}$ to be a random vector lying in the range space of $Q_2^T$, $d=Bx$, and $b=Ax+r$ with $r$ being a random vector of 2-norm $\omega$, i.e., $\|r\|_2=\omega$.  The vector of the Lagrange multipliers $-\lambda$ is given by equation \eqref{agmt.eq}. By our construction, the desired EILS problem follows, and it is easy to check that the assumption \eqref{asump1} is always satisfied. All the numerical experiments are performed in Matlab R2014a on a PC with Intel i5-6600M CPU 3.30 GHz and 4.00 GB RAM.

\begin{example}
\label{example1}
\rm
In this example, we will show that by choosing different $L$s the conditioning of some particular linear transformations of the solution can be easily obtained, and give a comparison of the normwise, mixed and componentwise condition numbers.

Let $p=20$, $q=10$, $n=20$, $s=5$, $\omega=10^{-9}$, $\kappa(A)\leq 10$ and $\kappa(B)=1$. We set $L_1=I_{20}$, $L_2^T=[I_3,\mathbf{0}]\in \mathbb{R}^{3\times 20}$ which means the first three elements in the solution are of interest, and $L_3^T\in \mathbb{R}^{1\times 20}$ with nonnegative elements and $\|L_3\|_1=1$ which is a convex combination of the solution. For reproducibility, we set the random number stream as \texttt{t=RandStream('mt19937ar','Seed',2018)}. Once the data is generated, we multiply the last three columns of $A$ by $10^{-\tau}$ and the first three rows of $B$ by $10^\tau$. To compare the error bounds given by different condition numbers, we  give random perturbations with a preassigned magnitude to the coefficients as follows
\begin{align*}
\Delta A= \epsilon*E\circ A,\;\Delta B= \epsilon*F\circ B,\; \Delta b= \epsilon*g\circ b,\textrm{ and }\Delta d= \epsilon*h\circ d,
\end{align*}
where $\epsilon=10^{-9}$ and the elements in $E,F,g,h$ are generated from the uniform distribution on interval $[-1,1]$. The EILS problem  and its perturbed version are solved by the GQR-Cholesky method \cite{Boja03} which is simple but backward stable. A more efficient and complicated method can be found in \cite{Mast14b}. We also introduce the following notation to measure the normwise, mixed and componentwise relative errors
 $$\small r_2=\frac{\left\|L^T(\hat{x}-x)\right\|_2}{\left\|L^Tx\right\|_2},\quad r_{\infty}^m=\frac{\left\|L^T(\hat{x}-x)\right\|_{\infty}}{\left\|L^Tx\right\|_{\infty}}, \quad r_{\infty}^c=\left\|\frac{L^T(\hat{x}-x)}{L^Tx}\right\|_{\infty}, $$
 $$\small \kappa_{2}^{bd}=\kappa_{2L\F2}\Delta_{rF}, \quad \kappa_{\infty}^{mbd}=\kappa_{\infty{L\F}}^{m}\Delta_{r\max}, \;\;\mathrm{ and} \;\; \kappa_{\infty}^{cbd}=\kappa_{\infty{L\F}}^{c}\Delta_{r\max}$$
with
\begin{equation*}
  \Delta_{rF}=\frac{\left\|(\Delta A,\Delta B,\Delta b,\Delta d)\right\|_F}{\left\|(A,B,b,d)\right\|_F}\;  \mathrm{ and }\;\Delta_{r\max}=\left\|\frac{(\Delta A,\Delta B,\Delta b,\Delta d)}{(A,B,b,d)}\right\|_{\max}.
\end{equation*}
The numerical results are reported in Table \ref{Table1}.
\begin{table}[htp]
\footnotesize
  \centering
\begin{tabular}{llrrrrrrrrr}
  \hline
    &   & $r_2$ & $\kappa_{2L\F2}$ & $\kappa_{2}^{bd}$ &  $r_{\infty}^m$ &  $\kappa_{\infty{L\F}}^{m}$  & $\kappa_{\infty}^{mbd}$ &  $r_{\infty}^c$  & $\kappa_{\infty{L\F}}^{c}$  & $\kappa_{\infty}^{cbd}$ \\
    \hline
   $\tau=0$ & $L1$  & 2.24e-09 & 1.09e03 & 1.28e-07 & 2.36e-09 & 40.8  & 2.35e-08 &1.55e-08  &618.73 & 3.55e-07\\
            & $L2$  & 2.65e-09 & 188.33  & 2.20e-08 & 3.43e-09 & 96.99 & 5.57e-08 &3.61e-09  &150.35 & 8.64e-08\\
            & $L3$  & 1.15e-09 & 19.39   & 2.27e-09 & 1.15e-09 & 34.61 & 1.99e-08 &1.15e-09  &34.61  & 1.99e-08\\
  \hline
   $\tau=4$ & $L1$  & 4.91e-09 & 1.10e07 & 8.82e-04 & 6.67e-09 & 91.14 & 5.24e-08 &3.89e-08  &572.20 & 3.29e-07\\
            & $L2$  & 5.88e-09 & 7.79e05 & 6.23e-05 & 5.73e-09 & 86.69 & 4.98e-08 &8.80e-09  &135.54 & 7.79e-08\\
            & $L3$  & 4.53e-09 & 1.49e05 & 1.19e-05 & 4.53e-09 & 58.56 & 3.36e-08 &4.53e-09  &58.56  & 3.36e-08\\
  \hline
\end{tabular}
\caption{First order relative forward error bounds based on the projected condition numbers}\label{Table1}
\end{table}
From Table \ref{Table1}, we can find that by varying the projection matrix $L$ the conditioning of a linear transformation of the solution can be easily obtained. When $\tau=0$, we note that all these three condition number based error bounds are very tight in most cases. With respect to the normwise relative error, we can find that the first three elements and the convex combination of the solution tend to be well conditioned for its smaller condition numbers and tighter error bounds. When $\tau=4$ which means the coefficient matrices are badly scaled,  the normwise condition number based error bound can largely overestimate the true error. Meanwhile, the error bounds based on mixed and componentwise condition numbers are still very tight, which shows that these error bounds are less sensitive to the componentwise perturbation and the scaling in the data. And in our example we also note that the scaling gives very little influence on the projected mixed and componentwise condition numbers.
\end{example}

\begin{example}
\label{example2}
\rm
Since the 2-norm projected condition number \eqref{eq.unsmpcd} contains Kronecker product which makes computing the exact value of the condition number expensive, Theorem \ref{thm.closedform} presents two compact forms of \eqref{eq.unsmpcd}. In this example, we give a comparison of the running time for computing the exact value of the 2-norm projected condition number via its three different forms. The computation is carried out with a ``naive" method, that is, we first formulate the matrices in \eqref{eq.smpcd1}, \eqref{eq.smpcd2} and \eqref{eq.unsmpcd} explicitly, and then compute its spectral norms with Matlab command \texttt{norm(*,2)}. The ``naive" method is usually preferred by the practitioners from applied disciplines.

We keep the residual $r$ with $\|r\|_2=10^{-6}$, $p/q=2$, $\kappa(B)=1$, and $\kappa(A)\leq 10$.  The parameters $\Phi$, $\Psi$, $\beta$, $\vartheta$  and $\xi_L$ are set to be $1$s. We set the projection matrix $L=I_n$ and use 100 replications for each setting. Since the three different expressions give the same value of normwise condition number, we only report the mean of CPU time in second in Table \ref{Table2}.
\begin{table}[htp]
  \centering
  \footnotesize
  \begin{tabular}{crrrrr}
     \hline
                 $(m,n,s)$      &    $(240,120,80)$ & $(360,180,120)$ & $(480,240,160)$ & $(600,300,200)$& $(960,480,320)$\\
     \hline
%      GQR-Cholesky solver       &    0.0160 & 0.0209 &0.0412  &0.0615 & 0.0999\\
 %    \hline
      $\kappa_{2L\F }(A,B,b,d)$ &    0.6770 & 2.6279 &6.8551 &14.0333& *\\
      $\kappa_{2L\F1}(A,B,b,d)$ &    0.0039 & 0.0091 &0.0304 &0.0571 & 0.1284\\
      $\kappa_{2L\F2}(A,B,b,d)$ &    0.0089 & 0.0246 &0.0522 &0.1047 & 0.2072\\
     \hline
   \end{tabular}
\caption{CPU time comparison of computing the exact value of 2-norm projected condition number via its three different forms}\label{Table2}
\end{table}
The numerical results show that direct computation of $\kappa_{2L\F }(A,B,b,d)$ is very time consuming compared with $\kappa_{2L\F1}(A,B,b,d)$ and $\kappa_{2L\F2}(A,B,b,d)$. When we raise $(m,n,s)$ to $(960,480,320)$, the computation of $\kappa_{2L\F }(A,B,b,d)$ breaks down due to the lack of memory. But both $\kappa_{2L\F1}(A,B,b,d)$ and $\kappa_{2L\F2}(A,B,b,d)$ can still be quickly computed. The main reason is that the Kronecker product enlarges the order of matrix, and this leads to a heavy computational burden and needs large storage space. Since the compact forms $\kappa_{2L\F1}(A,B,b,d)$ and $\kappa_{2L\F2}(A,B,b,d)$ eliminate the Kronecker product, it require very little storage space and can be efficiently computed. Moreover, we can also note that the computation of $\kappa_{2L\F1}(A,B,b,d)$ requires the least CPU time compared with the other two equivalent forms. This coincides with our Remark \ref{rmk3} that the matrix in $\kappa_{2L\F1}(A,B,b,d)$ has the smallest size. Thus we may say that finding compact form of the 2-norm condition number of the EILS problem is one possible way to improve the computational efficiency of calculating its exact value.
\end{example}

\begin{example}
\label{example3}
\rm
Theorem \ref{thm.infnorm} presents the explicit expressions of the projected mixed and componentwise condition numbers of the EILS problem. The Kronecker product in these expressions also makes computing its exact values expensive. As we have claimed that it is hard to find its equivalent and compact forms, we present upper bounds of the projected mixed and componentwise condition numbers in Remark \ref{rmk3.6}. In this example we will check the computational efficiency and tightness of these upper bounds.

Firstly, we give a comparison of the running time for computing the projected mixed and componentwise condition numbers and its upper bounds with its explicit expressions. In this case we set $L=I_n$, $\|r\|_2=10^{-7}$, $p/q=2$, $\kappa(B)=1$, and $\kappa(A)\leq 10$. For different sizes of the random EILS problems, we generate 100 independent data sets, and report the mean of the CPU time in Table \ref{Table3}.
\begin{table}[htp]
\footnotesize
  \centering
  \begin{tabular}{crrrr}
    \hline
    $(m,n,s)$ & $(240,120,80)$ & $(360,180,120)$ & $(480,240,160)$ & $(600,300,200)$ \\
    \hline
    $\kappa_{\infty{L\F}}^{m}(A,B,b,d)$ &0.2063 &0.6919 &1.7746 & 3.5572  \\
    $\kappa_{\infty\F}^{mU}(A,B,b,d)$   &0.0032 &0.0092 &0.0195 & 0.0555  \\
    \hline
  \end{tabular}
  \caption{CPU time comparison of computing mixed condition number and its upper bound}\label{Table3}
\end{table}
Since the CPU time of computing the projected mixed condition number and its upper bound is similar to that of projected componentwise condition number and its upper bound, Table~\ref{Table3} only contains the CPU time comparison of the projected mixed condition number and its upper bound.  Table~\ref{Table3} shows that the upper bound of the projected mixed condition number can be efficiently computed.

Then, to measure the tightness of the upper bounds we define the following ratios
\begin{equation*}
r_{m}=\frac{\kappa_{\infty\F}^{mU}(A,B,b,d)}{\kappa_{\infty{L\F}}^{m}(A,B,b,d)},\quad r_{c}=\frac{\kappa_{\infty\F}^{cU}(A,B,b,d)}{\kappa_{\infty{L\F}}^{c}(A,B,b,d)}.
\end{equation*}
The closer the ratios approach to 1, the tighter the upper bounds are. Here, we set $L=I_n$, $p=300$, $q=120$, $n=210$, $s=140$ and $\|r\|_2=10^{-4}$. The numerical results are summarized in Table~\ref{Table4}.
\begin{table}[htp]
\footnotesize
  \centering
  \begin{tabular}{llrrrr}
    \hline
    &$(\kappa(A) ,\kappa(B))$  & $(\approx10,10)$ & $(\approx10^6,10)$ & $(\approx10,10^6)$ & $(\approx10^6,10^6)$ \\
      \hline
    \multirow{3}*{$r_m$}& $\max$            &1.0000   &1.2759   &1.0000   &1.2703  \\
                        & $\mathrm{median}$ &1.0000   &1.2273   &1.0000   &1.2222  \\
                        & $\min$            &1.0000   &1.1751   &1.0000   &1.1430  \\
    \hline
     \multirow{3}*{$r_m$}& $\max$            &1.0000   &1.2769   &1.0000   &1.2913  \\
                         & $\mathrm{median}$ &1.0000   &1.2373   &1.0000   &1.2295  \\
                         & $\min$            &1.0000   &1.1665   &1.0000   &1.1725  \\
    \hline
  \end{tabular}
\caption{Ratios between projected mixed and componentwise condition numbers and its upper bounds}\label{Table4}
\end{table}
From Table~\ref{Table4}, we note that the upper bounds are nearly optimal when $A$ is well conditioned. When $A$ is ill conditioned, the upper bounds are still very tight. Thus, according to Tables \ref{Table3} and \ref{Table4}, when the forward error is bounded by the projected mixed and componentwise condition numbers multiplied by the backward error, it is preferable to use the upper bounds of these condition numbers. This is because these upper bounds are not only very tight but also can be efficiently computed.
\end{example}

\begin{remark}
\rm
Computing the exact value of the condition number with its explicit expression is of much interest to the practitioners from applied disciplines. The numerical experiments show the great computational efficiency of the new forms and upper bounds of the projected condition numbers. It should be noted that the size of random EILS problem used in this paper is not very large. For large EILS problem, computing the exact value of the projected condition numbers will be expensive, and some easy computable estimates are preferred. As far as we know, the condition number estimation methods can be divided into two branches. One is randomized method. For example, Li and Wang \cite{LiW16aX} proposed to use probabilistic spectral norm estimator \cite{Hochs13} to estimate the 2-norm condition number, and small-sample condition estimation method \cite{Gudmun95} to estimate the mixed and componentwise condition numbers of the ILS problem. These methods can be easily adapted to estimate the projected condition numbers of the EILS problem. The other is deterministic method. The classical power method \cite[Ch.~15]{Hig02} can be used to estimate the 2-norm projected condition number and the upper bounds of the projected mixed and componentwise condition numbers. The corresponding algorithms can be easily derived similar to \cite{WangL17TLS} and \cite{DiaoS16} in estimating the condition numbers of the TLS problem. These condition number estimation methods have been well developed and can be adapted to our settings without any technical difficulty. Thus we give little attention to the condition estimation of the EILS problem in the present paper.

\end{remark}

\section{Concluding remark}
In this paper, with the projected condition number we give a unified treatment of the condition numbers of the EILS problem. The main utility of the projected condition number is that it can be easily used to give the conditioning of a linear transformation of the solution. Moreover,  the projected condition number of the EILS problem includes some widely used condition numbers, like normwise, mixed and componentwise condition numbers, as its special cases. Considering practical applications and computation, we present the explicit expressions of the 2-norm projected condition number, and the projected mixed and componentwise condition numbers of the EILS problem. To reduce the computational burden in calculating the exact value of the 2-norm projected condition number, we present two compact and equivalent forms. For the projected mixed and componentwise condition numbers, we give some tight and easy computable upper bounds. Numerical experiments are given to show that the new forms and upper bounds of the projected condition numbers require less storage space and can be more efficiently computed compared with its original forms.

\section*{Acknowledgement}The authors would like to give their sincere thanks to the anonymous referees and the handling Editor for their detailed and helpful comments that improved the presentation of their paper.

%%%%%%%%%%%%%%%%%%%%%%%%%%%%%%%%%%%%%%%%
%\section{Another section}
%A few more items worth reading before you prepare an ELA article.
%
%
%
%\begin{lemma}
%\label{lem}
%The off-diagonal entries of a $2 \times 2$ symmetric matrix
%are the same.
%\end{lemma}
%
%{\em Proof.}
%The proof follows from the structure of the matrix, as follows:
%\begin{equation}
%\label{below}
%A =  \left[ \begin{array}{lr} a  & b \\ b &  c  \end{array}
%\right]. ~~~~~~ \cvd
%\end{equation}
%
%% There are various ways to include postscript figures in a
%% LaTeX file. Consult the manual for details.
%% Below is a particular example with an encapsulated postscript file
%% Note that the epsfig package must be included in the preamble.
%
%% \begin{figure}[ht]
%% \begin{center}
%% \epsfig{file=circle.eps,height=38mm,width=38mm,clip=}
%% \caption{A circle!}
%% \end{center}
%% \label{fig}
%% \end{figure}
%
%\bigskip
%{\bf Acknowledgment.} We wish to thank everybody who has contributed
%papers to the Electronic Journal of Linear Algebra.

%%%%%%%%%%%%%%%%%%%%%%%%%%%%%%%%%%%%%%%%%%%%%%%%%%%%%%%%%%%%%

\end{document}